\documentclass[12pt]{article}
\usepackage{amssymb}
\usepackage{amsthm}
\usepackage{amsmath}
\usepackage{graphicx}
\usepackage{enumerate}

\DeclareMathOperator{\Con}{Con}

\newtheorem{theorem}{Theorem}[section]
\newtheorem{definition}[theorem]{Definition}
\newtheorem{lemma}[theorem]{Lemma}

\newtheorem{remark}[theorem]{Remark}
\newtheorem{example}[theorem]{Example}

\title{Sheffer operation in relational systems}
\author{Ivan~Chajda and Helmut~L\"anger}
\date{}
\begin{document}

\footnotetext{Support of the research by the Austrian Science Fund (FWF), project I~4579-N, and the Czech Science Foundation (GA\v CR), project 20-09869L, entitled ``The many facets of orthomodularity'', as well as by \"OAD, project CZ~02/2019, entitled ``Function algebras and ordered structures related to logic and data fusion'', and, concerning the first author, by IGA, project P\v rF~2020~014, is gratefully acknowledged.}

\maketitle

\begin{abstract}
The concept of a Sheffer operation known for Boolean algebras and orthomodular lattices is extended to arbitrary directed relational systems with involution. It is proved that to every such relational system there can be assigned a Sheffer groupoid and also, conversely, every Sheffer groupoid induces a directed relational system with involution. Hence, investigations of these relational systems can be transformed to the treaty of special groupoids which form a variety of algebras. If the Sheffer operation is also commutative then the induced binary relation is antisymmetric. Moreover, commutative Sheffer groupoids form a congruence distributive variety. We characterize symmertry, antisymmetry and treansitivity of binary relations by identities and quasi-identities satisfied by an assigned Sheffer operation. The concepts of twist-products of relational systems and of Kleene relational systems are introduced. We prove that every directed relational system can be embedded into a directed relational system with involution via the twist-product construction. If the relation in question is even transitive, then the directed relational system can be embedded into a Kleene relational system. Any Sheffer operation assigned to a directed relational system $\mathbf A$ with involution induces a Sheffer operation assigned to the twist-product of $\mathbf A$.
\end{abstract}



\section{Introduction}

Relational systems form one of the most general mathematical structures. Almost all structures appearing in algebra can be considered as relational structures. Such structures were studied for a long lime, see the pioneering work by J.~Riguet (\cite R) from 1948 containing elementary properties and constructions with binary relations and the paper by R.~Fraiss\'e (\cite F) from 1954. On the other hand, in contrast to publications in algebra, not so many of papers are devoted to relational systems. One of the reasons is that there are not so powerful tools for investigating relations as there are for algebras. This is also the reason why relational systems do not appear so often in applications both in mathematics and outside. One important application of relational systems are e.g.\ Kripke systems used in the formalization of several non-classical logical systems.

The authors introduced formerly several methods where relational systems are connected with various accompanying algebras and hence their properties can be transformed into algebraic language and the problems are solved by tools developed in general algebra. Let us mention e.g.\ \cite{CL13} and \cite{CL16a} where certain groupoids similar to directoids are assigned or \cite{CL16b} and \cite{CLS} where this approach is applied to relational systems equipped with a unary operation. For ternary relations, such an approach was used in \cite{CKL}. However, the spectrum of used algebraic tools is not restricted only to these more or less elementary cases, in \cite{BC} relational systems are treated similarly as residuated ordered sets. In the present paper we extend this list of used tools by the so-called Sheffer operation.

Remember that the Sheffer operation introduced by H.~M.~Sheffer (\cite S) in 1913 was used in Boolean algebras as a very successful tool since this operation can replace all other Boolean operations. Namely, every Boolean operation, both basic or derived, can be expressed by repeatedly using the Sheffer operation, see e.g.\ \cite B. In today terminology, the clone of Boolean functions is generated by the Sheffer operation. This has a surprising and very successful application in technology because in switching circles, in particular in computer processors, it suffices to use only one binary operation, namely the Sheffer one. Then the technology of production of such chips is much easier and cheaper than it was in the beginning of computer era when several parts of the computer were composed by at least two different kinds of diodes (e.g.\ one for conjunction and the other one for negation). As it was shown by the first author in \cite C, a Sheffer operation can be introduced not only in Boolean algebras but also in orthomodular lattices or even in ortholattices (see \cite B for these concepts). These algebras form an algebraic axiomatization of the logic of quantum mechanics. Since not all authors agreed that such lattices are suitable for modeling the propositional calculus of the logic of quantum mechanics, they recognized that disjunction in this logic not necessarily exists for all elements, i.e.\ that the supremum of two elements need not exist if the these elements are not orthogonal. Hence so-called orthomodular posets and orthoposets were introduced. This was the reason why the concept of Sheffer operation was transferred from ortholattices to orthomodular posets and orthoposets, or, more generally, to ordered sets with an involution or a complementation, see \cite{CK}.

The next natural step is to extend this method from posets to more general relational systems. In order to avoid difficulties with not everywhere defined operations and some other drawbacks, we consider so-called directed relational systems where the relation is reflexive and equipped with a unary involution operation. The authors show that also in this case a kind of Sheffer operation can be introduced and the corresponding groupoid characterizes the given relational system. The benefit of this method is two-fold. At first, we show that similarly as for Boolean algebras, using an assigned Sheffer operation we can conversely recover not only the involution but also the given binary relation. Hence the Sheffer operation reduces the type of the relational system and substitutes a binary relation by an everywhere defined operation. And secondly, we show that some basic properties of binary relations can be characterized with advantage by using this operation and that the Sheffer operation enables also more advanced constructions for relational systems not considered so far.

\section{Basic concepts}

The Sheffer operation was introduced by H.~M.~Sheffer (\cite S) in Boolean algebras. If $\mathbf B=(B,\vee,\wedge,{}',0,1)$ is a Boolean algebra and one defines
\[
x|y:=x'\vee y'
\]
then $|$ is just the Sheffer operation on $\mathbf B$. For our reasons, we define it as follows.

\begin{definition}\label{def1}
A {\em Sheffer operation} on a non-void set $A$ is a binary operation $|$ on $A$ satisfying the following identities:
\begin{eqnarray}
& & (x|y)|(x|x)\approx x,\label{equ1} \\
& & (x|y)|(y|y)\approx y.\label{equ2}
\end{eqnarray}
A {\em Sheffer groupoid} is a groupoid $(A,|)$ where $|$ is a Sheffer operation on $A$.
\end{definition}

Hence, the class of Sheffer groupoids forms a variety of algebras.

\begin{example}\label{ex1}
If $A:=\{a,b,c,d\}$ and the binary operation $|$ on $A$ is defined by
\[
\begin{array}{c|cccc}
| & a & b & c & d \\
\hline
a & a & c & d & c \\
b & c & b & d & c \\
c & a & b & d & c \\
d & a & b & d & c
\end{array}
\]
then $(A,|)$ is a Sheffer groupoid.
\end{example}

It is worth noticing that the Sheffer operation in a Boolean algebra satisfies the identities (\ref{equ1}) and (\ref{equ2}) and hence our new concept is sound.

An {\em antitone involution} on a lattice $(L,\vee,\wedge)$ is a unary operation $'$ on $L$ satisfying
\begin{enumerate}[(i)]
\item $x''\approx x$,
\item $x\leq y$ implies $y'\leq x'$
\end{enumerate}
for all $x,y\in L$.

The following lemma was shown for ortholattices in \cite C.

\begin{lemma}
Let $(L,\vee,\wedge,{}')$ be a lattice with an antitone involution. Then {\rm(i)} and {\rm(ii)} hold:
\begin{enumerate}[{\rm(i)}]
\item If $x|y:=x'\vee y'$ for all $x,y\in L$ then $(L,|)$ is a Sheffer groupoid.
\item If $x|y:=x'\wedge y'$ for all $x,y\in L$ then $(L,|)$ is a Sheffer groupoid.
\end{enumerate}
\end{lemma}

\begin{proof}
\
\begin{enumerate}[(i)]
\item Since $x|x\approx x'\vee x'\approx x'$, (\ref{equ1}) and (\ref{equ2}) are equivalent to
\begin{align*}
(x'\vee y')'\vee x'' & \approx x, \\
(x'\vee y')'\vee y'' & \approx y,
\end{align*}
respectively.
\item Since $x|x\approx x'\wedge x'\approx x'$, (\ref{equ1}) and (\ref{equ2}) are equivalent to
\begin{align*}
(x'\wedge y')'\wedge x'' & \approx x, \\
(x'\wedge y')'\wedge y'' & \approx y,
\end{align*}
respectively.
\end{enumerate}
\end{proof}

\begin{lemma}
Axioms {\rm(\ref{equ1})} and {\rm(\ref{equ2})} are independent.
\end{lemma}

\begin{proof}
if $A:=\{a,b\}$ and the binary operation $|$ on $A$ is defined by $x|y:=x$ for all $x\in A$ then $|$ satisfies (\ref{equ1}), but not (\ref{equ2}) since $(a|b)|(b|b)=a|b=a\neq b$, and if $A:=\{a,b,c\}$ and the binary operation $|$ on $A$ is defined by
\[
\begin{array}{c|ccc}
| & a & b & c \\
\hline
a & a & b & c \\
b & c & b & c \\
c & a & a & c
\end{array}
\]
then $|$ satisfies (\ref{equ2}), but not (\ref{equ1}) since $(a|b)|(a|a)=b|a=c\neq a$.
\end{proof}

Let us recall some concepts from theory of relations.

Let $A$ be a non-void set, $a,b\in A$, $R$ a binary relation on $A$ and $'$ a unary operation on $A$. We define
\begin{align*}
U(a,b) & :=\{x\in A\mid(a,x),(b,x)\in R\}, \\
L(a,b) & :=\{x\in A\mid(x,a),(x,b)\in R\}
\end{align*}
and call these set the {\em upper cone} and {\em lower cone} of $a$ and $b$ with respect to $R$, respectively. The relational system $\mathbf A=(A,R)$ is called {\em directed} if $U(x,y)\neq\emptyset$ and $L(x,y)\neq\emptyset$ for all $x,y\in A$. The operation $'$ is called {\em antitone} if $(x,y)\in R$ implies $(y',x')\in R$ and an {\em involution} on $\mathbf A$ if it is antitone and if it satisfies the identity $x''\approx x$. It can be shown that in a relational system $(A,R,{}')$ with involution, if $U(x,y)\neq\emptyset$ for all $x,y\in A$ then $L(x,y)\neq\emptyset$ for all $x,y\in A$ since $L(x,y)\approx(U(x',y'))'$, where $B':=\{b'\mid b\in B\}$ for every subset $B$ of $A$.

\begin{definition}
A {\em directed relational system with involution} is an ordered triple $(A,R,{}')$ consisting of a non-void set $A$, a binary relation $R$ on $A$ and a unary operation $'$ on $A$ satisfying the following conditions:
\begin{eqnarray}
& & R\text{ is reflexive},\label{equ4} \\
& & (A,R)\text{ is directed},\label{equ5} \\
& & '\text{ is an involution on }(A,R).\label{equ6}
\end{eqnarray}
\end{definition}

\section{Representation of relational systems by Sheffer \\
groupoids}

The following result shows how a Sheffer groupoid is connected with a directed relational system with involution.

\begin{theorem}\label{th2}
Let $\mathbf A=(A,|)$ be a Sheffer groupoid and define a unary operation $'$ on $A$ and a binary relation $R$ on $A$ by
\begin{align*}
x' & :=x|x\text{ for all }x\in A, \\
 R & :=\{(x,y)\in A^2\mid x'|y'=y\}. 
\end{align*}
Then $\mathbb R(\mathbf A):=(A,R,{}')$ is a directed relational system with involution, the so-called {\em directed relational system with involution induced by $\mathbf A$}.
\end{theorem}

\begin{proof}
Let $a,b\in A$. (\ref{equ1}) implies $x''\approx x$ and that $R$ is reflexive. (\ref{equ1}) and (\ref{equ2}) can be written in the equivalent form $(x|y)|x'\approx x$ and $(x|y)|y'\approx y$, respectively. If $(a,b)\in R$ then $a'|b'=b$ and hence $b|a=(a'|b')|a=a'$, i.e.\ $(b',a')\in R$ showing that $'$ is an involution on $(A,R)$. Since $(a'|b')|a=a'$ and $(a'|b')|b=b'$ we have $((a'|b')',a'),((a'|b')',b')\in R$ and hence $(a,a'|b'),(b,a'|b')\in R$, i.e.\ $a'|b'\in U(a,b)$ which shows $U(a,b)\neq\emptyset$ proving that $(A,R)$ is directed.
\end{proof}

\begin{example}
$(A,A^2\setminus\{(a,b),(b,a)\},{}')$ where $a'=a$, $b'=b$, $c'=d$ and $d'=c$ is the directed relational system induced by the Sheffer groupoid $\mathbf A$ from Example~\ref{ex1}.
\end{example}

In the following we show that also conversely, to every directed relational system with involution a Sheffer groupoid can be assigned.

Let $\mathbf A=(A,R,{}')$ be a directed relational system with involution. Define a binary operation $|$ on $A$ as follows: Put $x|y:=y'$ if $(x',y')\in R$ and let $x|y$ be an arbitrary element of $U(x',y')$ otherwise {\rm(}$x,y\in A${\rm)}. Then $|$ will be called an {\em operation assigned to $\mathbf A$}.

\begin{lemma}\label{lem1}
Let $\mathbf A=(A,R,{}')$ be a directed relational system with involution and $|$ a binary operation on $A$. Then $|$ is assigned to $\mathbf A$ if and only if
\begin{enumerate}[{\rm(i)}]
\item $(x,y)\in R$ if and only if $x'|y'=y$,
\item $x|y\in U(x',y')$ for all $x,y\in A$.
\end{enumerate}
\end{lemma}

\begin{proof}
\item Let $a,b\in A$. First assume $|$ to be assigned to $\mathbf A$. If $(a,b)\in R$ then $(a'',b'')\in R$ and hence $a'|b'=b''=b$. Conversely, assume $a'|b'=b$. Then $(a,b)\notin R$ would imply $(a'',b'')\notin R$ and hence $b=a'|b'\in U(a'',b'')=U(a,b)$ and hence $(a,b)\in R$, a contradiction. Hence $(a,b)\in R$. This shows (i). If $(a',b')\in R$ then $a|b=b'\in U(a',b')$. Otherwise, $a|b\in U(a',b')$, too. This shows (ii). Conversely, if $|$ satisfies (i) and (ii) then clearly $|$ is assigned to $\mathbf A$.
\end{proof}

It should be remarked that if $(A,R,{}')$ is a directed relational system with involution and $|$ an assigned operation then condition (ii) of Lemma~\ref{lem1} is equivalent to
\[
(x|y)|(x|y)\in L(x,y)\text{ for all }x,y\in A.
\]

In the following we will often use this lemma. Now we prove the converse of Theorem~\ref{th2}.

\begin{theorem}\label{th1}
Let $\mathbf A=(A,R,{}')$ be a directed relational system with involution and $|$ an operation assigned to $\mathbf A$. Then $|$ is a Sheffer operation, a so-called {\em Sheffer operation assigned to $\mathbf A$}, i.e.\ $\mathbb G(\mathbf A):=(A,|)$ is a Sheffer groupoid, a so-called {\em Sheffer groupoid assigned to $\mathbf A$}.
\end{theorem}

\begin{proof}
Let $a,b\in A$. Since $(x',x')\in R$ we have $x|x\approx x'$. If $(a',b')\in R$ then $(b,a)\in R$ and hence $(a|b)|a'=b'|a'=a$ and $(a|b)|b'=b'|b'=b$. If $(a',b')\notin R$ then $a|b\in U(a',b')$ and hence $(a',a|b),(b',a|b)\in R$ which implies $((a|b)',a),((a|b)',b)\in R$, i.e.\ $(a|b)a'=a$ and $(a|b)b'=b$.
\end{proof}

\begin{remark}
In general, $\mathbb G(\mathbf A)$ is not uniquely defined. However, it contains all the information on the directed relational system $\mathbf A$ with involution. In other words, the given directed relational system with involution can be completely recovered from an assigned Sheffer groupoid, see the following result.
\end{remark}

\begin{theorem}\label{th3}
Let $\mathbf A=(A,R,{}')$ be a directed relational system with involution. Then $\mathbb R(\mathbb G(\mathbf A))=\mathbf A$.
\end{theorem}

\begin{proof}
If
\begin{align*}
           \mathbb G(\mathbf A) & =(A,|), \\
\mathbb R(\mathbb G(\mathbf A)) & =(A,S,{}^*)
\end{align*}
then according to the proof of Lemma~\ref{lem1},
\begin{align*}
  S & =\{(x,y)\in A^2\mid  x'|y'=y\}=\{(x,y)\in A^2\mid(x,y)\in R\}=R, \\
x^* & \approx x|x\approx x'.
\end{align*}
\end{proof}

On the other hand, we can show for which pairs of elements a Sheffer operation assigned to $\mathbb R(A,|)$ coincides with the Sheffer operation $|$ of a given Sheffer groupoid $(A,|)$.

\begin{theorem}
Let $\mathbf A=(A,|)$ be a Sheffer groupoid and $\mathbb G(\mathbb R(\mathbf A))=(A,\circ)$. Then $x\circ y=x|y$ if $x|y=y|y$.
\end{theorem}

\begin{proof}
If $\mathbb R(\mathbf A)=(A,R,{}')$ then any of the following assertions implies the next one:
\begin{align*}
     x|y & =y|y, \\
     x|y & =y', \\
 (x',y') & \in R, \\
x\circ y & =y', \\
x\circ y & =x|y.
\end{align*}
\end{proof}

In fact, $\circ$ need not coincide with $|$ as can be seen by the following example.

\begin{example}
If $|$ is the Sheffer operation from Example~\ref{ex1} then $\circ$ has the operation table
\[
\begin{array}{c|cccc}
\circ & a & b & c & d \\
\hline
  a   & a & x & d & c \\
  b   & y & b & d & c \\
  c   & a & b & d & c \\
  d   & a & b & d & c
\end{array}
\]
where $x,y\in\{c,d\}$ since $U(a,b)=\{c,d\}$ in the induced relational system. Hence, if we take $x=d$ or $y=d$ then $\circ$ differs from $|$.
\end{example}

We have shown that directed relational systems with involution are nearly in a one-to-one correspondence with Sheffer groupoids. Analogously as for Boolean algebras where the Sheffer operation substitutes all other operations since they can be derived from it, also here the Sheffer operation substitutes both the binary relation and the unary operation. Hence it enables us to reduce the type of the directed relational system with involution.

\section{Elementary properties of relations}

In the following we characterize some of properties of the relation $R$ of a directed relational system $\mathbf A=(A,R,{}')$ with involution by means of identities and quasi-identities for a Sheffer operation assigned to $\mathbf A$.

\begin{theorem}
Let $\mathbf A=(A,R,{}')$ be a directed relational system with involution and $|$ an assigned Sheffer operation. Then $R$ is symmetric if and only if $|$ satisfies the identity
\begin{equation}
((x|y)|(x|y))|x\approx x|x.\label{equ7}
\end{equation}
\end{theorem}

\begin{proof}
If $R$ is symmetric then any of the following assertions implies the next one:
\begin{align*}
            x|y & \in U(x',y'), \\
       (x',x|y) & \in R, \\
       (x|y,x') & \in R, \\
       (x|y)'|x & \approx x', \\
((x|y)|(x|y))|x & \approx x|x.
\end{align*}
If, conversely, $|$ satisfies identity (\ref{equ7}) then any of the following assertions implies the next one:
\begin{align*}
(x,y) & \in R, \\
x'|y' & =y, \\
y'|x' & =(x'|y')'|x'=x, \\
(y,x) & \in R.
\end{align*}
\end{proof}

Another important property of a binary relation is antisymmetry. Recall that a binary relation $R$ is antisymmetric if $(x,y),(y,x)\in R$ implies $x=y$.

\begin{theorem}\label{th4}
Let $\mathbf A=(A,R,{}')$ be a directed relational system with involution and $|$ a Sheffer operation assigned to it. Then the following hold:
\begin{enumerate}[{\rm(i)}]
\item $R$ is antisymmetric if and only if $x|y=y'$ and $y|x=x'$ imply $x=y$.
\item If $x|y\approx y|x$ then $R$ is antisymmetric.
\end{enumerate}
\end{theorem}

\begin{proof}
\
\begin{enumerate}[(i)]
\item is clear.
\item This follows from (i) since $x|y\approx y|x$, $x|y=y'$ and $y|x=x'$ imply $x=(y|x)'=(x|y)'=y$.
\end{enumerate}
\end{proof}

Transitivity of a binary relation can be expressed by an identity for an assigned Sheffer operation as follows.

\begin{theorem}
Let $\mathbf A=(A,R,{}')$ be a directed relational system with involution and $|$ an assigned Sheffer operation. Then $R$ is transitive if and only if $|$ satisfies the identity
\begin{equation}
x|(((x|y)|(x|y))|z)|(((x|y)|(x|y))|z)\approx((x|y)|(x|y))|z.\label{equ8}
\end{equation}
\end{theorem}

\begin{proof}
If $R$ is transitive then any of the following assertions implies the next one:
\begin{align*}
                      x|y\in U(x',y') & \text{ and }(x|y)'|z\in U(x|y,z'), \\
              (x',x|y),(x|y,(x|y)'|z) & \in R, \\
                        (x',(x|y)'|z) & \in R, \\
                        x|((x|y)'|z)' & \approx (x|y)'|z, \\
x|(((x|y)|(x|y))|z)|(((x|y)|(x|y))|z) & \approx((x|y)|(x|y))|z.
\end{align*}
If, conversely, $|$ satisfies identity (\ref{equ8}) then any of the following assertions implies the next one:
\begin{align*}
(x,y),(y,z) & \in R, \\
    x'|y'=y & \text{ and }y'|z'=z, \\
      x'|z' & =x'|(y'|z')'=x'|((x'|y')'|z')'=(x'|y')'|z'=y'|z'=z, \\
      (x,z) & \in R.
\end{align*}
\end{proof}

Let us introduce the following concepts. A {\em bounded relational system with involution} is an ordered quintuple $\mathbf A=(A,R,{}',$ $0,1)$ such that $(A,R,{}')$ is a directed relational system with involution, $0,1\in A$ and $(0,x),(x,1)\in R$ hold for all $x\in A$. $\mathbf A$ is called {\em complemented} if it is bounded and if $U(x,x')\approx1\approx0'$. In such a case $L(x,x')\approx0$. Also these properties of relational systems can be characterized  by identities and quasi-identities for an assigned Sheffer operation.

\begin{theorem}
Let $(A,R,{}')$ be a directed relational system with involution and $|$ a Sheffer operation assigned to it. Moreover, let $0,1\in A$ and put $\mathbf A:=(A,R,{}',0,1)$. Then the following hold:
\begin{enumerate}[{\rm(i)}]
\item $\mathbf A$ is bounded if and only if it satisfies the identities $(0|0)|x\approx x|x$ and $x|(1|1)\approx1$.
\item $\mathbf A$ is complemented if it is bounded, $0|0\approx1$ and if for every $x,y\in A$,
\[
x|(y|y)=(x|x)|(y|y)=y\text{ implies }y=1.
\]
\end{enumerate}
\end{theorem}

\begin{proof}
\
\begin{enumerate}[(i)]
\item The assertions $(0,x')\in R$ and $(x',1)\in R$ are equivalent to $0'|x\approx x'$ and $x|1'\approx1$, respectively.
\item The following are equivalent:
\begin{align*}
           y & \in U(x,x'), \\
(x,y),(x',y) & \in R, \\
        x|y' & =x'|y'=y, \\
     x|(y|y) & =(x|x)|(y|y)=y.
\end{align*}
\end{enumerate}
\end{proof}

As mentioned in Section~2, the class of Sheffer groupoids forms a variety $\mathcal V$. We can ask one more condition, namely commutativity of $|$. As shown in Theorems~\ref{th3} and \ref{th4}, the directed relational systems with involution induced by commutative Sheffer groupoids will have antisymmetric binary relations. We present a subvariety of $\mathcal V$ containing all commutative Sheffer groupoids which has an important congruence property.

We recall that a variety $\mathcal V$ of algebras is called {\em congruence distributive} if every member   of $\mathcal V$ has a distributive congruence lattice.

\begin{theorem}
The variety of Sheffer groupoids $(A,|)$ satisfying the identities
\begin{eqnarray}
& & (x|y)|(x|x)\approx(x|x)|(x|y),\label{equ3} \\
& & (x|y)|(y|y)\approx(y|y)|(x|y)\label{equ9}
\end{eqnarray}
is congruence distributive.
\end{theorem}

\begin{proof}
If $x':=x|x$ and $m(x,y,z):=((x|y)|(x|z))'|(y|z)$ then
\begin{align*}
m(x,z,z) & \approx((x|z)|(x|z))'|(z|z)\approx(x|z)|z'\approx z\text{ by (\ref{equ2})}, \\
m(x,y,x) & \approx((x|y)|(x|x))'|(y|x)\approx x'|(y|x)\approx x\text{ by (\ref{equ1}), (\ref{equ9}) and (\ref{equ2})}, \\
m(x,x,z) & \approx((x|x)|(x|z))'|(x|z)\approx x'|(x|z)\approx x\text{ by (\ref{equ3}) and (\ref{equ1})}.
\end{align*}
\end{proof}

\section{Kleene relational systems and twist-products}

At first, we show how homomorphisms of Sheffer groupoids are related with homomorphisms of induced directed relational systems with involution. Because in the literature there are different concepts of homomorphism of relational systems, we recall the following one.

Let $(A,R)$ and $(B,S)$ be relational systems. A mapping $f:A\rightarrow B$ is called a {\em homomorphism} from $(A,R)$ to $(B,S)$ if
\[
(x,y)\in R\text{ implies }(f(x),f(y))\in S.
\]
A {\em homomorphism} f is called {\em strong} if
\[
(x,y)\in R\text{ if and only if }(f(x),f(y))\in S.
\]
If $(A,R,{}')$ and $(B,S,^*)$ are relational systems with unary operation then $f$ is a {\em homomorphism} from $(A,R,{}')$ to $(B,S,^*)$ if it is a homomorphism from $(A,R)$ to $(B,S)$ satisfying
\[
f(x')=(f(x))^*\text{ for all }x\in A.
\]

\begin{theorem}\label{th7}
Let $\mathbf A=(A,|_A)$ and $\mathbf B=(B,|_B)$ be Sheffer groupoids and $f$ a homomorphism from $\mathbf A$ to $\mathbf B$. Then $f$ is a homomorphism between the induced directed relational systems $\mathbb R(\mathbf A)$ and $\mathbb R(\mathbf B)$ with involution.
\end{theorem}

\begin{proof}
Let $a,b\in A$, $\mathbb R(\mathbf A)=(A,R,{}')$ and $\mathbb R(\mathbf B)=(B,S,{}^*)$. We have $f(x')\approx f(x|_Ax)\approx f(x)|_Bf(x)\approx(f(x))^*$ and hence any of the following assertions implies the next one:
\begin{align*}
              (a,b) & \in R, \\
            a'|_Ab' & =b, \\
         f(a'|_Ab') & =f(b), \\
      f(a')|_Bf(b') & =f(b), \\
(f(a))^*|_B(f(b))^* & =f(b), \\
        (f(a),f(b)) & \in S.
\end{align*}
\end{proof}

For the converse direction, we firstly mention the following result for bounded relational systems.

\begin{lemma}
Let $(A,R,{}',0_A,1_A)$ and $(B,S,^*,0_B,1_B)$ be bounded relational systems with involution and $f$ a strong homomorphism from $\mathbf A=(A,R,{}')$ to $\mathbf B=(B,S,^*)$. Further assume that $f(1_A)=1_B$. Define binary operations $|_A$ and $|_B$ on $A$ and $B$, respectively, by
\[
x|_Ay:=\left\{
\begin{array}{ll}
y'  & \text{if }(x',y')\in R, \\
1_A & \text{otherwise}
\end{array}
\right.
\quad\quad x|_By:=\left\{
\begin{array}{ll}
y^* & \text{if }(x^*,y^*)\in S, \\
1_B & \text{otherwise}
\end{array}
\right.
\]
Then $(A,|_A)$ and $(B,|_B)$ are Sheffer groupoids assigned to $\mathbf A$ and $\mathbf B$, respectively, and $f$ is a homomorphism from $(A,|_A)$ to $(B,|_B)$.
\end{lemma}

\begin{proof}
Let $a,b\in A$. Obviously, $(A,|_A)$ and $(B,|_B)$ are Sheffer groupoids. If $(a',b')\in R$ then $((f(a))^*,(f(b))^*)=(f(a'),f(b'))\in S$ and hence $f(a|_Ab)=f(b')=(f(b))^*=f(a)|_Bf(b)$. If $(a',b')\notin R$ then $((f(a))^*,(f(b))^*)=(f(a'),f(b'))\notin S$ and hence $f(a|_Ab)=f(1_A)=1_B=f(a)|_Bf(b)$.
\end{proof}

We are going to determine conditions under which the converse of Theorem~\ref{th7} holds.

\begin{theorem}
Let $\mathbf A=(A,R,{}')$ and $\mathbf B=(B,S,{}^*)$ be directed relational systems with involution, $f$ a strong surjective homomorphism from $\mathbf A$ to $\mathbf B$ and $|_A$ a Sheffer operation assigned to $\mathbf A$ and assume that the equivalence relation $\ker f$ on $A$ is a congruence on $(A,|_A)$. Then there exists a Sheffer operation $|_B$ on $B$ such that $f$ is a homomorphism from $(A,|_A)$ to $(B,|_B)$ and $|_B$ is assigned to $\mathbf B$.
\end{theorem}

\begin{proof}
Define $f(x)|_Bf(y):=f(x|_Ay)$ for all $x,y\in A$. Since $\ker f\in\Con(A,|_A)$, $|_B$ is well-defined. Let $a,b\in A$. Then any of the following assertions implies the next one:
\begin{align*}
((f(a))^*,(f(b))^*) & \in S, \\
      (f(a'),f(b')) & \in S, \\
            (a',b') & \in R, \\
              a|_Ab & =b', \\
           f(a|_Ab) & = f(b'), \\
        f(a)|_Bf(b) & =(f(b))^*.
\end{align*}
Moreover, any of the following assertions implies the next one:
\begin{align*}
((f(a))^*,(f(b))^*) & \notin S, \\
      (f(a'),f(b')) & \notin S, \\
            (a',b') & \notin R, \\
              a|_Ab & \in U(a',b'), \\
           f(a|_Ab) & \in U(f(a'),f(b')), \\
        f(a)|_Bf(b) & \in U((f(a))^*,(f(b))^*).
\end{align*}
This shows that $|_B$ is a Sheffer operation on $B$ assigned to $\mathbf B$. According to the definition of $|_B$ we have $f(x|_Ay)=f(x)|_Bf(y)$ for all $x,y\in A$.
\end{proof}

For a lattice $\mathbf L=(L,\vee,\wedge)$ its twist-product $(L^2,\sqcup,\sqcap)$ is defined by
\begin{align*}
(x,y)\sqcup(z,v) & :=(x\vee z,v\wedge y), \\
(x,y)\sqcap(z,v) & :=(x\wedge z,v\vee y)
\end{align*}
for all $(x,y),(z,v)\in L^2$. We extend this concept to relational systems as follows.

Let $A$ be a non-void set and $R$ a binary relation on $A$. Then $(A^2,S,{}^*)$ with
\begin{align*}
      S & :=\{((x,y),(z,v))\in(A^2)^2\mid(x,z),(v,y)\in R\}, \\
(x,y)^* & :=(y,x)
\end{align*}
for all $(x,y)\in A^2$ will be called the {\em twist-product of $(A,R)$}.

Recall that an {\em embedding} of a relational system $\mathbf A$ into a relational system $\mathbf B$ is an injective strong homomorphism from $\mathbf A$ to $\mathbf B$.

The importance of twist-products is illuminated by the next result.

\begin{theorem}\label{th5}
Let $\mathbf A=(A,R)$ be a relational system, $a\in A$ and $\mathbf B=(A^2,S,{}^*)$ the twist-product of $\mathbf A$. Then the following hold:
\begin{enumerate}[{\rm(i)}]
\item If $\mathbf A$ is directed then $\mathbf B$ is a directed relational system with involution $^*$,
\item the mapping $x\mapsto(x,a)$ is an embedding of $\mathbf A$ into $(A^2,S)$.
\end{enumerate}
\end{theorem}

\begin{proof}
Let $a,b,c,d\in A$.
\begin{enumerate}[(i)]
\item Assume $\mathbf A$ to be directed. Since $(a,a),(b,b)\in R$ we have $((a,b),(a,b))\in S$ showing reflexivity of $S$. Because of $U((a,b),(c,d))=U(a,c)\times L(b,d)$, $(A^2,S)$ is directed. Moreover, $(x,y)^{**}\approx(y,x)^*\approx(x,y)$, and the following are equivalent:
\begin{align*}
    ((a,b),(c,d)) & \in S, \\
      (a,c),(d,b) & \in R, \\
      (d,b),(a,c) & \in R, \\
    ((d,c),(b,a)) & \in S, \\
((c,d)^*,(a,b)^*) & \in S.
\end{align*}
Hence, $^*$ is an involution on $(A^2,S)$.
\item The mapping $x\mapsto(x,a)$ is injective. Moreover, $((b,a),(c,a))\in S$ if and only if $(b,c)\in R$.
\end{enumerate}
\end{proof}

Hence, every directed relational system can be embedded into a directed relational system with involution.

The question arises whether a Sheffer operation assigned to the twist-product of a directed relational $\mathbf A$ with involution can be derived from a Sheffer operation assigned to $\mathbf A$. We give a positive answer in the following theorem.

\begin{theorem}
Let $(A,R,{}')$ be a directed relational system with involution, $|_A$ an assigned Sheffer operation on $A$ and define
\[
(x,y)|_B(z,v):=(y'|_Av',(x|_Az)')
\]
for all $(x,y),(z,v)\in A^2$. Then $|_B$ is a Sheffer operation on $A^2$ assigned to the twist-product of $(A,R)$.
\end{theorem}

\begin{proof}
For $a,b,c,d\in A$ the following are equivalent:
\begin{align*}
 ((a,b)^*,(c,d)^*) & \in S, \\
     ((b,a),(d,c)) & \in S, \\
       (b,d),(c,a) & \in R, \\
 (b'',d''),(a',c') & \in R, \\
   (b'|_Ad',a|_Ac) & =(d'',c'), \\
(b'|_Ad',(a|_Ac)') & =(d,c), \\
     (a,b)|_B(c,d) & =(d,c), \\
     (a,b)|_B(c,d) & =(c,d)^*
\end{align*}
and the following are equivalent:
\begin{align*}
   (b'|_Ad',a|_Ac) & \in U(b'',d'')\times U(a',c'), \\
(b'|_Ad',(a|_Ac)') & \in U(b,d)\times L(a,c), \\
(b'|_Ad',(a|_Ac)') & \in U((b,a),(d,c)), \\
     (a,b)|_B(c,d) & \in U((a,b)^*,(c,d)^*).
\end{align*}
\end{proof}

In order to simplify notation we extend binary relations between elements of a non-void set $A$ to relations between subsets of $A$.

Let $A$ be a non-void set, $b,c$ be elements of $A$, $B,C$ be subsets of $A$ and $R$ be a binary relation on $A$. We say $(B,C)\in R$ if $B\times C\subseteq R$. Instead of $(\{b\},C)\in R$ and $(B,\{c\})\in R$ we shortly write $(b,C)\in R$ and $(B,c)\in R$, respectively.

The concept of a Kleene lattice was introduced by J.~A.~Kalman (\cite K). Recall that a distributive lattice $(L,\vee,\wedge,{}')$ with antitone involution is called a {\em Kleene lattice} if it satisfies the so-called {\em normality condition}, i.e.\ the identity
\[
x\wedge x'\leq y\vee y'\text{ for all }x,y\in L.
\]
These lattices are used in logic in order to formalize certain De Morgan propositional logics. For posets with involution, this notion was already generalized by the authors in \cite{CL} in the following way: A distributive poset $(P,\leq,{}')$ with involution is called a {\em Kleene poset} if
\[
L(x,x')\leq U(y,y')\text{ for all }x,y\in P
\]
which means that $z\leq v$ for all $x,y\in P$ and all $(z,v)\in L(x,x')\times U(y,y')$.

\begin{definition}
\
\begin{enumerate}[{\rm(i)}]
\item A {\em Kleene relational system} is a relational system $(A,R,{}')$ with an antitone involution satisfying
\[
(L(x,x'),U(y,y'))\in R\text{ for all }x,y\in A.
\]
\item If $\mathbf A=(A,R)$ is a relational system, $a\in A$ and $(A^2,S,$ ${}^*)$ the twist-product of $\mathbf A$ then we define the following subset of $A^2$:
\[
P_a(\mathbf A):=\{(x,y)\in A^2\mid(L(x,y),a),(a,U(x,y))\in R\}.
\]
\end{enumerate}
\end{definition}

It is worth noticing that Kleene lattices and Kleene posets are Kleene relational systems according to our previous definition.

Using the above defined subset of the twist-product, we can show that every directed relational system with a transitive relation can be embedded into a Kleene relational system.

\begin{theorem}
Let $\mathbf A=(A,R)$ be a directed relational system, $a\in A$, $(A^2,S,{}^*)$ the twist-product of $\mathbf A$ and $T:=S\cap(P_a(\mathbf A))^2$. Then the following hold:
\begin{enumerate}[{\rm(i)}]
\item If $R$ is transitive then $(P_a(\mathbf A),T,{}^*)$ is a directed relational system with involution which is a Kleene relational system,
\item the mapping $x\mapsto(x,a)$ is an embedding of $\mathbf A$ into $(P_a(\mathbf A),T)$.
\end{enumerate}
\end{theorem}

\begin{proof}
Let $(b,c),(d,e)\in P_a(\mathbf A)$.
\begin{enumerate}[(i)]
\item Put $\mathbf B:=(P_a(\mathbf A),T,{}^*)$. From $(b,c)\in P_a(\mathbf A)$ we conclude $(L(b,c),a),(a,U(b,c))\in R$ and hence $(L(c,b),a),(a,U(c,b))\in R$, i.e.\ $(b,c)^*=(c,b)\in P_a(\mathbf A)$ which shows that $P_a(\mathbf A)$ is closed with respect to $^*$. According to Theorem~\ref{th5}, $\mathbf B$ is a directed relational system with involution. Because of
\begin{align*}
(L((b,c),(c,b)),(a,a))=(L(b,c)\times U(b,c),(a,a)) & \in S, \\
((a,a),U(d,e)\times L(d,e))=((a,a),U((d,e),(e,d))) & \in S
\end{align*}
we have $(L((b,c),(c,b)),U((d,e),(e,d)))\in S$ due to transitivity of $S$ (which follows from the transitivity of $R$) and hence $\mathbf B$ is a Kleene relational system.
\item For all $x\in A$ we have $(L(x,a),a),(a,U(x,a))\in R$ and hence $(x,a)\in P_a(\mathbf A)$. The rest follows from Theorem~\ref{th5}.
\end{enumerate}
\end{proof}

It should be remarked that if $R$ is transitive then $(P_a(\mathbf A),T,{}^*)$ is a relational subsystem of the twist-product $(A^2,S,{}^*)$ of $\mathbf A$.

Authors' addresses:

Ivan Chajda \\
Palack\'y University Olomouc \\
Faculty of Science \\
Department of Algebra and Geometry \\
17.\ listopadu 12 \\
771 46 Olomouc \\
Czech Republic \\
ivan.chajda@upol.cz

Helmut L\"anger \\
TU Wien \\
Faculty of Mathematics and Geoinformation \\
Institute of Discrete Mathematics and Geometry \\
Wiedner Hauptstra\ss e 8-10 \\
1040 Vienna \\
Austria, and \\
Palack\'y University Olomouc \\
Faculty of Science \\
Department of Algebra and Geometry \\
17.\ listopadu 12 \\
771 46 Olomouc \\
Czech Republic \\
helmut.laenger@tuwien.ac.at

\end{document}